\documentclass[12pt,reqno]{amsart}
\usepackage{amsmath,amssymb,amscd}
\usepackage[all]{xy}
\usepackage[latin1]{inputenc}
\usepackage[T1]{fontenc}
\usepackage{hyperref}
\usepackage{ifthen}
\usepackage{fullpage}
\usepackage{subsupscripts}

\newboolean{workmode}
\setboolean{workmode}{false}

\newboolean{sections}
\setboolean{sections}{true}

\input xy
\xyoption{all}

\def\eoe{\unskip\ \hglue0mm\hfill$\diamond$\smallskip\goodbreak}

\makeatletter
\def\th@plain{%
  \thm@notefont{}
  \itshape 
}
\def\th@definition{%
  \thm@notefont{}
  \normalfont 
}
\makeatother

\newcommand{\calF}{{\mathcal{F}}}

\newcommand{\calK}{{\mathcal{K}}}
\newcommand{\calO}{{\mathcal{O}}}
\newcommand{\calP}{{\mathcal{P}}}
\newcommand{\calQ}{{\mathcal{Q}}}
\newcommand{\calR}{{\mathcal{R}}}

\newcommand{\CC}{{\mathbb{C}}}  
\newcommand{\PP}{{\mathbb{P}}}  
\newcommand{\RR}{{\mathbb{R}}}  
\renewcommand{\SS}{{\mathbb{S}}} 
\newcommand{\ZZ}{{\mathbb{Z}}}  

\newcommand{\depth}{{\operatorname{depth}}} 
\newcommand{\Stab}{{\operatorname{Stab}}} 

\renewcommand{\O}{{\operatorname{O}}} 

\newcommand{\CIN}{{C^\infty}}   

\newcommand{\hypref}[2]{{\hyperref[#1]{#2~\ref{#1}}}}

\newcommand{\ifwork}[1]{\ifthenelse{\boolean{workmode}}{#1}{}}
\newcommand{\comment}[1]{}
\newcommand{\mute}[1]{}
\newcommand{\printname}[1]{}

\ifwork{
\renewcommand{\comment}[1]{{\marginpar{*}\ \scriptsize{#1}\ }}
\renewcommand{\printname}[1]
    {\smash{\makebox[0pt]{\hspace{-1.0in}\raisebox{8pt}{\tiny #1}}}}
}

\newcommand{\labell}[1] {\label{#1} \printname{#1}}

\newcommand{\ifsection}[2]{\ifthenelse{\boolean{sections}}{#1}{#2}}

\theoremstyle{plain}
\ifsection{
    \newtheorem{theorem}{Theorem}[section]
}
{
    \newtheorem{theorem}{Theorem}
}

\newtheorem{proposition}[theorem]{Proposition}

\newtheorem{corollary}[theorem]{Corollary}
\newtheorem{lemma}[theorem]{Lemma}

\theoremstyle{definition}
\newtheorem{definition}[theorem]{Definition}
\newtheorem{example}[theorem]{Example}
\newtheorem{remark}[theorem]{Remark}

\newtheorem*{corollary*}{Corollary}

\title{Orbit Spaces of Linear Circle Actions}

\author{Suzanne Craig}
\address{Department of Mathematics, University of Colorado Boulder, Boulder, Colorado, USA 80309}
\email{Suzanne.Craig@colorado.edu}

\author{Naiche Downey}
\address{University of Colorado Boulder, Boulder, Colorado, USA 80309}
\email{Naiche.Downey@colorado.edu}

\author{Lucas Goad}
\address{University of Colorado - Colorado Springs, Colorado Springs, Colorado, USA 80918}
\email{lgoad@uccs.edu}

\author{Michael J.\ Mahoney}
\address{122 E Bayaud Ave, Denver, Colorado, USA 80209}
\email{Michael.J.Mahoney@colorado.edu}

\author{Jordan Watts}
\address{Department of Mathematics, Central Michigan University, Mt. Pleasant, Michigan, USA 48859}
\email{jordan.watts@cmich.edu}

\date{\today}

\begin{document}

\begin{abstract}
In this paper, it is shown that non-isomorphic effective linear circle actions yield non-diffeomorphic differential structures on the corresponding orbit spaces.
\end{abstract}

\maketitle

\section{Introduction}\label{s:intro}

Recall that an orbifold is a topological space equipped with an atlas of linear representations of finite groups; in the case that all of these representations are effective, we say that the orbifold is effective (see, for instance, one of \cite{haefliger,MP} for the precise definition).  One can equip an effective orbifold with a ``smooth structure'' in many different ways \cite{MP,lerman,IKZ,watts-orb}.  No matter which notion of smoothness is taken, in \cite{watts-orb} it is shown that the underlying local semi-algebraic set of a smooth (effective) orbifold, equipped with its natural differential structure, holds a complete set of orbifold invariants in its differential structure; that is, an atlas for the orbifold can be recovered from the smooth functions on the orbifold alone.  It is natural to ask what happens in the case of a quotient by a smooth circle action on a manifold.  The purpose of this paper is to take the first step toward solving this problem by considering the case of linear circle actions: can one recover a linear circle action (up to diffeomorphism) by examining the differential structure of the orbit space alone?  

The above question and result can be seen from a broader perspective: there is a functor from Lie groupoids to differential spaces, sending a groupoid to its orbit space \cite{watts-gpds}.  Studying this functor, especially when restricted to proper Lie groupoids, leads to a modern connection between two classical subjects: Lie group actions of compact groups, and singular spaces (namely, semi-algebraic varieties).  The result on orbifolds in \cite{watts-orb} is that this functor when restricted to proper effective \'etale Lie groupoids is essentially injective (\emph{i.e.} injective up to isomorphism).  This paper deals with the restriction to linear $\SS^1$-actions.

The generalization of smooth structures from manifolds to arbitrary subspaces and quotient spaces has a long history, though the perspective we take was first formally defined by Sikorski \cite{sikorski67,sikorski71} and is presented in brief here (see \cite{sniatycki} for more details on these spaces).

\begin{definition}[Differential Space] \label{d:diff space}
Let $X$ be a non-empty set.  A \textbf{differential structure} on $X$ is a non-empty family
$\calF$ of real-valued functions satisfying:
\begin{enumerate}
\item (Smooth Compatibility) For any positive integer $k$, functions
$f_1,...,f_k\in\calF$, and $g\in\CIN(\RR^k)$, the composition
$g(f_1,...,f_k)$ is contained in $\calF$.
\item (Locality) Equip $X$ with the initial topology induced by $\calF$; that is, the weakest topology such that each function in $\calF$ is continuous.  Let $f\colon X\to\RR$ be a function such that for any $x\in X$
there exist an open neighborhood $U$ of $x$ and a function
$h\in\calF$ satisfying $f|_U=h|_U$.  Then $f\in\calF$.
\end{enumerate}
A set $X$ equipped with a differential structure $\calF$ is called
a \textbf{differential space}, and is denoted $(X,\calF)$.  We will drop the notation $\calF$ when it is superfluous. In the literature (for example, \cite{schwarz}), authors use differential structures without naming them, possibly unaware that the structures had been formally named. 
\end{definition}

\begin{definition}[Smooth Maps Between Differential Spaces]
Let $(X,\mathcal{F}_X)$ and $(Y,\mathcal{F}_Y)$ be two differential
spaces.  A map $F \colon X\to Y$ is \textbf{smooth}
if $F^*\calF_Y\subseteq\calF_X$.  $F$ is called a
\textbf{diffeomorphism} if it
is smooth and has a smooth inverse.
Denote the set of smooth maps between $X$ and $Y$ by $\calF(X,Y)$.
\end{definition}

Differential spaces with smooth maps between them form a category closed under taking subsets and quotients.

\begin{definition}[Subspace Differential Structure]
Let $(X,\calF)$ be a differential space, and let $Y\subseteq X$ be
a subset.  Then $Y$ acquires a differential structure $\calF_Y$ as follows: $f\in\calF_Y$ if for every $y\in Y$ there exist an open neighborhood $U\subseteq X$ of $y$ and a function $g\in\calF$ such that $f|_{U\cap Y}=g|_{U\cap Y}$.  We call $(Y,\calF_Y)$ a \textbf{differential subspace} of $(X,\calF_X)$.  Note that the subspace topology on $Y$ equals the initial topology induced by $\calF_Y$; see \cite[Lemma 2.28]{watts-phd}.
\end{definition}

\begin{definition}[Quotient Differential Structure]
Let $(X,\calF)$ be a differential space, let $\sim$ be an
equivalence relation on $X$, and let $\pi \colon X \to X/\!\!\sim$
be the quotient map.  Then $X/\!\!\sim$ obtains a differential
structure $\calF_\sim$, called the \textbf{quotient differential structure},
consisting of those functions
$f \colon X/\!\!\sim\ \to\RR$
whose pullback $\pi^*f \colon X \to \RR$ is in $\calF$.
\end{definition}

It follows from a famous result of Schwarz \cite{schwarz} that if $G$ is a compact Lie group acting effectively and orthogonally on $\RR^m$, then the orbit space $\RR^m/G$ embeds into a Euclidean space $\RR^n$ so that the quotient differential structure on $\RR^m/G$ equals the subspace differential structure induced by the embedding.  Denote this differential structure by $\CIN(\RR^m/G)$.  In the case of finite groups, the diffeomorphism class of $(\RR^m/G,\CIN(\RR^m/G))$ as a differential space determines the group $G$ up to isomophism, and the $G$-representation up to isomorphism (see the proof of the Main Theorem of \cite{watts-orb}).  The invariants obtained from $\CIN(\RR^m/G)$ can be interpreted through two different perspectives: either quotient (linear representation) invariants on $\RR^m\to\RR^m/G$, or as subset invariants of $(\RR^m/G,\CIN(\RR^m/G))$, arising by examining the underlying semi-algebraic set of the orbit space (see Subsection~\ref{ss:lin repr} for a definition of semi-algebraic set).

Moreover, this recovery of linear actions of finite groups allows one to form an orbifold atlas of a general effective orbifold from the differential structure of its orbit space.  Consequently, the diffeomorphism class of the orbit space of a proper, effective, and locally-free Lie group action on a manifold determines the corresponding Lie groupoid up to Morita equivalence.  What is interesting about this result is that it does not hold for all proper and effective Lie group actions on manifolds, not even in the linear case.  Indeed, consider $\O(m)$ acting on $\RR^m$ by rotations and reflections.  The orbit space is diffeomorphic to the closed ray $[0,\infty)$, which is independent of $m$, and so the subset invariants do not form a complete set of invariants for the group actions.  And so, the original action cannot be recovered.  The question remains, however: what information is missing from the subset invariants that would lead to a complete set of invariants so that the action can be recovered from the orbit space?

In this paper, we examine the case of a linear circle action, and obtain the following theorem.

\begin{theorem}\label{t:main}
Let $\SS^1$ act linearly and effectively on $\RR^m$.  The diffeomorphism class of the quotient differential space $(\RR^m/\SS^1,\CIN(\RR^m/\SS^1))$ determines the $\SS^1$-action up to equivariant linear isomorphism.
\end{theorem}

Since linear isomorphisms are examples of diffeomorphisms, this theorem answers the question in the first paragraph above affirmatively.  Note that one needs to know that the differential space \emph{came from} a \emph{linear} $\SS^1$-action in the theorem above.  Even knowing that the space is an orbit space of a smooth effective $\SS^1$-action on a manifold is not sufficient to recover the action: consider $\SS^1$ acting (effectively) on $\SS^2$ by rotation about a fixed axis.  The orbit space is diffeomorphic to the manifold-with-boundary $[-1,1]$.  But this action descends through the antipodal action to $\RR\PP^2$, whose orbit space $[0,1]$ is diffeomorphic to $[-1,1]$.  The missing information that would distinguish between these two orbit spaces is knowledge about the isotropy groups: at the pre-image of $-1\in\SS^2/\SS^1$, this isotropy group is $\SS^1$; whereas at the pre-image of $0\in\RR\PP^2/\SS^1$, it is $\ZZ/2\ZZ$.  Characterizing this information is the subject of future work.

This paper is organized as follows.  In Section~\ref{s:prelims}, we review compact group actions, orbit-type stratifications, the result on orbifolds mentioned above, and linear circle actions.  In Section~\ref{s:invt poly}, we give a description of $\SS^1$-invariant polynomials for a linear circle action, culminating in Corollary~\ref{c:invt poly}.  In Section~\ref{s:orbit space}, we take the opportunity to study the structure of the orbit space of a linear $\SS^1$-action.  While this is not needed in the proof of the main result, this structure is interesting in its own right, and important for understanding the singularities that arise.  Section~\ref{s:linear case} contains the proof of Theorem~\ref{t:main}, and Section~\ref{s:examples} contains several examples.

\subsection*{Acknowledgements}
This paper is a result of a summer Research Experience for Undergraduates (REU) project in 2016, hosted by the Department of Mathematics at the University of Colorado Boulder.  The authors wish to thank the department for their hospitality and support, and the anonymous referee for excellent suggestions.

\section{Preliminaries}\label{s:prelims}

\subsection{Linear Compact Group Actions}\label{ss:lin repr}

Let us begin by reviewing the result of Schwarz, and related background material.  Fix a compact Lie group $G$, and assume $G$ is acting linearly and effectively on $\RR^m$.  Without loss of generality, since $G$ is compact, we can assume that $G\subseteq\O(m)$.  Denote by $P(\RR^m)$ the ring of real-valued polynomials on $\RR^m$, and by $P(\RR^m)^G$ the subring of polynomials invariant under the $G$-action; that is, polynomials $p$ satisfying $p(g\cdot x)=p(x)$ for all $x\in\RR^m$ and $g\in G$.  By Hilbert's Basis Theorem, $P(\RR^m)^G$ is finitely-generated.  That is, there exist $\sigma_1,\dots,\sigma_n\in P(\RR^m)^G$ such that for any $p\in P(\RR^m)^G$, there exists a polynomial $q\in P(\RR^n)$ such that $p=q(\sigma_1,\dots,\sigma_n)$.  Moreover, we can choose $\sigma_1,\dots,\sigma_n$ all to be homogeneous; that is, for each $\sigma_i$, its terms are all of the same degree.

Geometrically, what this means is that we can form the \textbf{Hilbert map} $\sigma\colon\RR^m\to\RR^n$ defined to be the $n$-tuple $(\sigma_1,\dots,\sigma_n)$, and for every $p\in P(\RR^m)^G$ there exists $q\in P(\RR^n)$ such that $p=q\circ\sigma$.  That is, $P(\RR^m)^G$ is the image of the \textbf{pullback map} $\sigma^*\colon P(\RR^n)\to P(\RR^m)$ sending $q$ to $\sigma^*q:= q\circ\sigma$.

Schwarz \cite{schwarz} extends this result from polynomials to smooth functions: the image of $\sigma^*\colon\CIN(\RR^n)\to\CIN(\RR^m)$ is exactly the invariant smooth functions $\CIN(\RR^m)^G$, the set of all smooth functions $f\in\CIN(\RR^m)$ such that $f(g\cdot x)=f(x)$ for all $x\in\RR^m$ and $g\in G$.

Let $\pi\colon\RR^m\to\RR^m/G$ be the quotient map.  Schwarz further shows that $\sigma$ descends to a topological embedding $i\colon\RR^m/G \to \RR^n$; hence we can view $\RR^m/G$ as a subset of $\RR^n$, and it obtains a subspace differential structure in this way.  Since $\pi^*$ is an isomorphism from $\CIN(\RR^m/G)$ to $\CIN(\RR^m)^G$, which in turn is isomorphic to $\sigma^*\CIN(\RR^n)$, we conclude that $i^*\CIN(\RR^n)=\CIN(\RR^m/G)$; that is, the subspace differential structure equals the quotient differential structure.

We have the following lemma, which we use in the sequel.

\begin{lemma}\labell{l:topologies}
Let $G$ be a compact group acting on a smooth manifold $M$.  The initial topology induced by $\CIN(M/G)$ equals the quotient topology on the orbit space.
\end{lemma}

\begin{proof}
Let $U$ be an open set in the initial topology on $M/G$, and fix $x\in U$.  There is a function $f\in\CIN(M/G)$ such that
$$x\in f^{-1}((0,1)) \subseteq U.$$
Let $\pi\colon M\to M/G$ be the quotient map, and $\widetilde{f}\in\CIN(M)$ be such that $\pi^*f = \widetilde{f}.$
Then $\pi^{-1}(f^{-1}((0,1)))=\widetilde{f}^{-1}((0,1))$, and so is open in $M$.  So, $f^{-1}((0,1))$ is open in the quotient topology.  It follows that $U$ is in the quotient topology.

In the other direction, let $U$ be an open set in the quotient topology.  Fix $x\in U$.  Then the orbit $\pi^{-1}(x)$ is closed, and contained in the $G$-invariant open set $\pi^{-1}(U)$.  Let $\widetilde{b}\colon M\to[0,1]$ be a smooth bump function equal to $1$ on the orbit $\pi^{-1}(x)$ with support in $\pi^{-1}(U)$.  After averaging over $G$, we may assume $\widetilde{b}$ is $G$-invariant.  Thus it descends to a smooth function $b$ on $M/G$ such that 
$$x\in b^{-1}((0,\infty))\subseteq U.$$
That is, $U$ is in the initial topology.
\end{proof}

In particular, there is no ambiguity in the topology on $\RR^m/G$ that we use. We can describe how $\RR^m/G$ sits in $\RR^n$ as a subset.  A \textbf{semi-algebraic set} $S=\cup_{i=1}^{m}S_{i}$ is a subset of $\RR^n$ where the subsets $S_i$ are of the form
$$S_{i}=\{x\in\RR^n\mid r_{i,1}(x),\dots,r_{i,k_i}(x)>0 \text{ and } s_{i,1}(x)=\dots=s_{i,\ell_i}(x)=0\},$$ where $r_{i,1},\dots,r_{i,k_i},s_{i,1},\dots,s_{i,\ell_i}\in P(\RR^n)$.  For our purposes, we will assume $S$ is equipped with the subspace differential structure induced by $\RR^n$.  (We call a differential space that is locally diffeomorphic to semi-algebraic sets a \textbf{local} semi-algebraic set.)  The Tarski-Seidenberg Theorem \cite{seidenberg,tarski48,tarski98} states that the image of a semi-algebraic set under a polynomial map (such as the Hilbert map above) is again a semi-algebraic set.  It follows that since $\RR^m/G$ sits inside $\RR^n$ as the image of $\sigma$, it is a semi-algebraic set.

\subsection{Compact Group Actions on Manifolds}\label{ss:cpct gp actions}

We now want to extend these ideas to group actions on manifolds.  Again, let $G$ be a compact Lie group acting smoothly on a manifold $M$ with quotient map $\pi\colon M\to M/G$.  Let $x\in M$, and let $H$ be the stabilizer of the action at $x$.  Note that $H$ is compact.  Define the \textbf{isotropy action} of $H$ on $T_xM$ by $h\cdot v=h_*v$ for any $v\in T_xM$.  Here we view elements of $G$ as diffeomorphisms of $M$, and since elements of $H$ fix $x$, this action is well-defined.  It is also linear, which in the effective case puts us back into the situation described above.  Note that for any $v\in T_x(G\cdot x)$, any smooth curve $c\colon\RR\to G\cdot x$ such that $c(0)=x$, and $\dot{c}(0)=v$, and any $h\in H$,
$$h\cdot v = h_*v = \frac{d}{dt}\Big|_{t=0}h\cdot c(t)$$
where $\frac{d}{dt}\Big|_{t=0}h\cdot c(t)$ is in $T_x(G\cdot x)$ since it is the derivative of a new smooth curve contained in $G\cdot x$.  Thus, $T_x(G\cdot x)$ is an $H$-invariant linear subspace of $T_xM$, and so the isotropy action descends to a linear $H$-action (also called the \textbf{isotropy action}) on the normal space $V:=T_xM/T_x(G\cdot x)$ to the $G$-orbit at $x$.

Since $H$ is a subgroup of $G$, it acts on $G$ by $h\cdot g=gh^{-1}$, and so we have the $H$-action on the product $G\times V$ defined by $$h\cdot(g,v):=(gh^{-1},h\cdot v).$$  Denote the orbit space of this action by $G\times_H V$.  Note that $G$ acts on $G\times_H V$ by $g'\cdot [g,v]=[g'g,v]$.   We have the following theorem of Koszul \cite{koszul}, \cite[Section 2.3]{DK}.

\begin{theorem}[The Slice Theorem]\label{t:slice}
Let $G$ be a compact Lie group acting on a manifold $M$.  For any $x\in M$ there exists an open $G$-invariant neighborhood $U$ of $x$, and a $G$-equivariant diffeomorphism $F\colon U\to G\times_HV$ where $H$ is the stabilizer of $x$ and $V$ is the normal space to $G\cdot x$ in $M$ at $x$ equipped with the isotropy action of $H$.
\end{theorem}

\begin{remark}\label{r:slice}
The Slice Theorem holds more generally for proper actions \cite{palais}, but we only need the compact case.
\end{remark}

It follows from the Slice Theorem that for any point $\pi(x)\in M/G$ there is an open neighborhood of $\pi(x)$ of the form $V/H\cong(G\times_H V)/G$, where $V$ is the normal space to $G\cdot x$ at $x$ as above.  By Lemma~\ref{l:topologies}, the quotient topology on $M/G$ is equal to the initial topology induced by the quotient differential structure $\CIN(M/G)$.  Since $V/H$ is semi-algebraic, it follows that $M/G$ is a local semi-algebraic set.

\subsection{Orbit-Type Stratification}\label{ss:ot strat}

Given a compact Lie group action of $G$ on $M$, for any closed subgroup $H\leq G$ we define the subset $M_{(H)}$ as follows: 
$$M_{(H)}:=\{x\in M\mid \exists g\in G \text{ s.t. } \Stab(x)=gHg^{-1}\}.$$
The connected components of these subsets partition $M$ into embedded submanifolds which together form a \textbf{stratification}; in particular, the partition is locally finite and if $C_1$ and $C_2$ are two such submanifolds such that $C_1\cap \overline{C_2}\neq\emptyset$, then either $C_1=C_2$ or $C_1$ is contained in the boundary of $C_2$.  We refer to this stratification as the \textbf{orbit-type stratification}.  (We do not intend to give the full definition of a stratification.  This is in fact very involved and will take us away from the point of the paper.  The reader who is interested should consult, for example, the book by Pflaum \cite{pflaum}.  For details on the orbit-type stratification, see \cite[Section 2.7]{DK}.)

The sets $M_{(H)}$ are $G$-invariant, and so descend to a partition of $M/G$ into subsets $M_{(H)}/G$.  The connected components of these again form a stratification, which again we will call the \textbf{orbit-type stratification} of $M/G$ (see \cite[Sections 2.7, 2.8]{DK}).  We are interested in the local form of these stratifications.  Fix $x\in M$.  By the Slice Theorem, there is a $G$-invariant open neighborhood $U$ of $x$ and a $G$-equivariant diffeomorphism $U\to G\times_H V$ where $H$ is the stabilizer of $x$ and $V$ is the normal space to $G\cdot x$ at $x$, equipped with the isotropy action.  As noted above, $(G\times_H V)/G$ is diffeomorphic to $V/H$.  Since $H$ is compact, there exists an $H$-invariant inner product on $V$, and with respect to this inner product we can write $V\cong E\oplus F$ where $E$ is the linear subspace of $H$-fixed points (that is, the maximal subspace on which $H$ acts trivially), and $F$ is an $H$-invariant complement.  It follows that $V/H\cong E\times(F/H)$.  Denote by $k$ be the dimension of $F$, by $\SS^{k-1}\subseteq F$ the unit sphere with respect to an $H$-invariant norm, and by $L$ the quotient $\SS^{k-1}/H$.  The continuous map $\SS^{k-1}\times[0,\infty)\to F$ sending $(x,t)\mapsto xt$ is $H$-invariant, and descends to a homeomorphism between the cone of $L$, given by $(L\times[0,\infty))/(L\times\{0\}),$ and the quotient $F/H$.  The cone itself is a stratified space, with a stratum $S\times(0,\infty)$ for each orbit-type stratum $S$ of $L$, along with the apex of the cone which we denote by $z$.  The stratification of $V/H$ contains a stratum $E\times S'$ for each stratum $S'$ of $F/H$.  We refer to $L$ as the \textbf{link} of this stratification, and the apex of the cone $z$ the \textbf{distinguished stratum} of $F/H$.  As a differential space, $F/H\smallsetminus\{z\}$ is diffeomorphic to $L\times(0,\infty)$; however, be aware that the differential structure of $F/H$ in any neighborhood of $z$ does not necessarily equal the quotient differential structure near the apex of $(L\times[0,\infty))/(L\times\{0\})$.  We explore the differential structure near $z$ via an example at the end of Section~\ref{s:orbit space}.

As a last word on orbit-type stratifications, we have the following theorem of \'Sniatycki \cite[Theorem 4.3.10]{sniatycki}.  (While \'Sniatycki's proof is only for compact connected groups, the proof goes through for any proper action.)

\begin{theorem}[The Orbit-Type Stratification is an Invariant]\label{t:ot strat}
Let $G$ be a compact Lie group acting smoothly on a manifold $M$.  Then the orbit-type stratification of $M/G$ is an invariant of $\CIN(M/G)$.
\end{theorem}

The proof of the above theorem comes from the fact that the connected components of orbit-type strata are exactly the 
accessible sets (also called orbits in the literature) of the family of all vector fields on $M/G$ induced by $\CIN(M/G)$.  The details of this would take us too far afield, and so we merely emphasize the fact that the stratification is an invariant of the differential structure.

\subsection{Recovering the Action - The Finite Group Case}\label{ss:orbifolds}

The purpose of this paper is to address the following question.  Given an effective linear $\SS^1$-action on $\RR^m$, can we recover the action from the differential space $(\RR^m/\SS^1,\CIN(\RR^m/\SS^1))$?  As mentioned in the introduction, in the case of a finite group $\Gamma$ acting effectively and linearly on $\RR^m$, the answer is affirmative: one can obtain invariants of the semi-algebraic set $(\RR^m/\Gamma,\CIN(\RR^m/\Gamma))$ from which the action of $\Gamma$ on $\RR^m$ can be recovered up to isomorphism.  Recall that a compact Lie group action of $G$ on an $m$-dimensional manifold $M$ is \textbf{locally free} if the stabilizer of every point is finite.  In this case, the Slice Theorem implies that for every point $x$ of $M/G$, there is a finite subgroup $\Gamma$, the \textbf{isotropy group} of $x$, a linear $\Gamma$-action on $\RR^m$, and an open neighborhood of $x$ diffeomorphic to $\RR^m/\Gamma$.  The orbit space $M/G$ equipped with an open cover by these neighborhoods, and for each of these neighborhoods the corresponding linear representation, is an \textbf{(effective) orbifold}.  (Again, we do not propose to give a rigorous definition of an orbifold; see the literature cited in the introduction for details.)  As mentioned above, locally, the differential structure encodes the linear representations of these finite groups; piecing together these local pictures into a global one, the following theorem due to the fifth author \cite{watts-orb} is obtained.

\begin{theorem}[The Orbifold Differential Structure]\label{t:orbifold}
Let $G$ be a compact Lie group acting smoothly, effectively, and locally freely on a connected manifold $M$.  The diffeomorphism class of the quotient differential space $(M/G,\CIN(M/G))$ determines the group action of $G$ on $M$, up to Morita equivalence.
\end{theorem}

\begin{remark}\label{r:orbifold}
\noindent
\begin{enumerate}

\item\label{i:orbifold1} We will not define Morita equivalence in this paper, as this requires the introduction of the language of Lie groupoids.  We simply note that the Morita equivalence class of a Lie groupoid representing the orbifold can be recovered from the differential structure.

\item\label{i:orbifold2} The ``definition'' of an effective orbifold we use above is not the typical definition used.  In the literature, one usually uses an atlas definition or Lie groupoids.  However, it is a theorem that any effective orbifold is the quotient of a manifold by a compact, effective, and locally free Lie group action.  See \cite{satake56}, \cite[Section 1.5]{satake57} \cite{haefliger}, \cite{moerdijk}, \cite[Theorem 4.1]{MP}.
\end{enumerate}
\end{remark}

\begin{corollary}\labell{c:orbifold}
 Let $X$ be an effective orbifold, and fix $x\in X$.  Then the isotropy group at $x$ is determined up to isomorphism by the differential structure of $X$.
\end{corollary}

While Corollary~\ref{c:orbifold} is stated here as a consequence of Theorem~\ref{t:orbifold}, this fact is actually used in part of the proof of Theorem~\ref{t:orbifold} (see \cite[Theorem 5.10]{watts-orb}).  To prove it, locally about $x$, one uses an algorithm of Haefliger and Ngoc Du \cite{HND} that reproduces the orbifold fundamental group (and thus the isotropy group at $x$) from knowledge of the codimension-0, codimension-1, and codimension-2 strata, and the orders of isotropy groups at these codimension-2 strata.  In turn, these orders of isotropy groups at codimension-2 strata can be obtained via the Milnor numbers of the corresponding singularities forming the codimension-2 strata.  For details, consult \cite{watts-orb}.

\subsection{Linear Circle Actions}\label{ss:lin circ actions}

Let $\SS^1$ act linearly on $\RR^n$.  There is an $\SS^1$-equivariant linear change of coordinates $\RR^n\cong\RR^{n-2m}\times\CC^m$ where $\SS^1$ acts trivially on $\RR^{n-2m}$, and on $\CC^m$ we have for all $e^{i\theta}\in\SS^1$ and $(z_1,\dots,z_m)\in\CC^m$, 
\begin{equation}\label{e:circle action}
e^{i\theta}\cdot(z_1,\dots,z_m):=(e^{i\theta\alpha_1}z_1,\dots,e^{i\theta\alpha_m}z_m);
\end{equation}
the numbers $\alpha_j$ are integers called the \textbf{weights} of the action.  Since complex conjugation is a diffeomorphism, we may assume the weights are non-negative; in fact, we may assume further that any weight-$0$ factor is included in the $\RR^{n-2m}$ factor, and so each $\alpha_j$ is positive.  Finally, if we also impose the condition that the action be effective, then $\gcd(\alpha_1,\dots,\alpha_m)=1$ and all isotropy actions are effective.

Assume $\SS^1$ acts on $\CC^m$ effectively with positive weights.
Denote by $S_{j_1\dots j_k}$ the subset
$$\{(0,\dots,0,z_{j_1},0,\dots,0,z_{j_k},0,\dots,0)\mid z_{j_\ell}\neq 0,~\ell=1,\dots,k\}.$$
The open dense subset $S_{1\dots m}$ has trivial stabilizer at all points since
$$\gcd(\alpha_1,\dots,\alpha_m)=1.$$
It is a submanifold of dimension $2m$, and there exists exactly one such submanifold.
The set $S_{1\dots \widehat{j}\dots m}$ is a submanifold of dimension $2m-2$, and there exist exactly $m\choose 1$ such submanifolds; the hat symbol means that we remove the corresponding index.  And so on.  If $e^{i\theta}$ fixes a point in $S_{j_1\dots j_k}$, then 
$e^{i\theta\alpha_{j_\ell}} = 1$
for $\ell = 1,\dots,k$.  That is, $e^{i\theta}\in\ZZ_{\alpha_{j_\ell}}$ for $\ell=1,\dots,k$, which is equivalent to $e^{i\theta}\in\ZZ_{\gcd(\alpha_{j_1},\dots,\alpha_{j_k})}.$
We tabulate this data in Table~\ref{t:s-sets}.
\begin{table}
\begin{tabular}{|c|c|c|c|}
\hline
\underline{Set} & \underline{Order of Stabilizer} & \underline{Codimension} & \underline{Number} \\
\hline
$S_{1\dots m}$ & $1=\gcd(\alpha_1,\dots,\alpha_m)$ & $0$ & $1 = {m\choose 0}$ \\
\hline
$S_{j_1\dots \widehat{j_\ell} \dots j_m}$ & $\gcd(\alpha_{j_1},\dots,\widehat{\alpha_{j_\ell}},\dots,\alpha_{j_m})$ & $2$ & $m\choose 1$ \\
\hline
\vdots & \vdots & \vdots & \vdots \\
\hline
$S_j$ & $\alpha_j$ & $2m-2$ & ${m\choose m-1}$ \\
\hline
$\{0\}$ & $\infty$ & $2m$ & $1= {m\choose m}$ \\
\hline
\end{tabular}
\caption{}
\label{t:s-sets}
\end{table}

One can also organize this table as integer labels on an $(m-1)$-simplex.  Noticing that the $m$ weights appear as the stabilizers of the sets $S_j$, we place each of these weights at the vertices of the simplex.  If an edge connects two vertices labeled $\alpha_j$ and $\alpha_k$, then we attach the integer label $\gcd(\alpha_j,\alpha_k)$ to the edge.  More generally, attach the integer label $\gcd(\alpha_{i_1},\dots,\alpha_{i_{\ell+1}})$ to an $\ell$-face whose vertices have associated weights $\alpha_{i_1},\dots,\alpha_{i_{\ell+1}}$.  The interior of the simplex obtains a label of $1$ since $\gcd(\alpha_1,\dots,\alpha_m)=1$.

The collection of sets in Table~\ref{t:s-sets} partitions $\CC^m$ into invariant submanifolds, and an orbit-type stratum is exactly the union of sets above whose points share the same stabilizer.

\section{Description of the Invariant Polynomials}\label{s:invt poly}

Our first order of business is to obtain a description of the invariant polynomials for an effective linear action of $\SS^1$ on $\CC^m$ as in Equation~\eqref{e:circle action}.  A fully satisfactory description of the algebra of invariant polynomials for a general linear circle action (such as a generating set of invariant polynomials with minimal cardinality along with the relations between these polynomials) remains elusive; at least the authors are not aware of such a description in the literature.  Often a minimal generating set can be obtained for specific cases, and on a case-by-case basis the relations can be derived from the invariant polynomials using a Gr\"obner basis (\cite[Chapter 1]{sturmfels}).  Also, work has been done on determining the dimension of the subspace of invariant polynomials of a fixed degree (see \cite{HS}).  We take a different approach below, in which we give a simple condition that any invariant polynomial must satisfy.  Fix an effective linear action of $\SS^1$ on $\CC^m$.  Considering $\CC^m$ as the real vector space $\RR^{2m}$, it will be convenient to use coordinates $(z_1,\overline{z}_1,\dots,z_m,\overline{z}_m)$.  Let $p$ be a homogeneous $\CC$-valued polynomial on $\CC^m$ of degree $d$.  Let $\calK$ be the set of all $2n$-tuples $K=(k_1,\overline{k}_1,\dots,k_n,\overline{k}_n)$ such that $k_1+\overline{k}_1+\cdots+k_n+\overline{k}_n=d$.  Then, $p$ takes the form
\begin{equation}\label{e:polynomial}
p(z_1,\overline{z}_1, \dots, z_n,\overline{z}_n) = \sum_{K\in\calK}P_Kz_1^{k_1}\overline{z}_1^{\overline{k}_1}\cdots z_n^{k_n}\overline{z}_n^{\overline{k}_n}
\end{equation}
for some complex numbers $P_K$.

\begin{proposition}\label{p:invt poly}
Let $\SS^1$ act on $\CC^m$ linearly and effectively with positive weights $\alpha_1,\dots,\alpha_m$.  Then a homogeneous $\CC$-valued polynomial $p$ as in Equation~\eqref{e:polynomial} is invariant if and only if it satisfies the equation
\begin{equation}\label{e:invt poly}
\alpha_1(k_1 - \overline{k}_1)+\dots+\alpha_m(k_m-\overline{k}_m) = 0
\end{equation}
for each $K\in\calK$ such that $P_K\neq 0$.
\end{proposition}

\begin{proof}
Fix a homogeneous polynomial as in Equation~\eqref{e:polynomial}.  Then $p$ takes the following form:
$$
p(z_1,\overline{z}_1, \dots, z_m,\overline{z}_m) = \sum_{K\in\calK}P_K|z_1|^{k_1+\overline{k}_1}\cdots|z_m|^{k_m+\overline{k}_m} e^{i\left(\psi_1(k_1-\overline{k}_1)+\cdots+\psi_m(k_m-\overline{k}_m)\right)}
$$
where $z_j=|z_j|e^{i\psi_j}$ for each $j$.  Applying $e^{i\theta}$ to $p$ for an arbitrary $e^{i\theta}\in\SS^1$, consider the difference $p-(e^{i\theta})^*p$:
\begin{equation}\label{e:invt poly2}
p(z_1,\dots,\overline{z}_m)-p(e^{i\theta}\cdot(z_1,\dots,\overline{z}_m))
\end{equation}
$$
=\sum_{K\in\calK}P_K|z_1|^{k_1+\overline{k}_1}\cdots|z_m|^{k_m+\overline{k}_m}e^{i\left(\psi_1(k_1-\overline{k}_1)+\cdots+\psi_m(k_m-\overline{k}_m)\right)} \left( 1- e^{i\theta\left(\alpha_1(k_1-\overline{k}_1)+\cdots+\alpha_m(k_m-\overline{k}_m)\right)}\right). 
$$
If $p$ satisfies Equation~\eqref{e:invt poly} for each $K$ such that $P_K\neq 0$, then the right-hand side of Equation~\eqref{e:invt poly2} is $0$, from which it follows that $p$ is invariant.

Conversely, assume $p$ is invariant. Then for any $e^{i\theta}\in\SS^1$ the two polynomials $p$ and $(e^{i\theta})^*p$ are equal; in particular, their polynomial coefficients are equal.  In terms of Equation~\ref{e:invt poly2}, this means that for each $K\in\calK$, 
$$P_K\left( 1- e^{i\theta\left(\alpha_1(k_1-\overline{k}_1)+\cdots+\alpha_m(k_m-\overline{k}_m)\right)}\right)=0.$$
Since $e^{i\theta}$ is arbitrary, it follows that for each $K\in\calK$, either $P_K=0$ or Equation~\ref{e:invt poly} holds.
\end{proof}

\begin{corollary}\label{c:invt poly}
Let $\SS^1$ act on $\CC^m$ linearly and effectively with positive weights $\alpha_1,\dots,\alpha_m$.  Then there exists a generating set of the invariant polynomials consisting solely of real and imaginary parts of monomials $z_1^{k_1}\overline{z_1}^{\overline{k_1}}\dots z_n^{k_n}\overline{z_n}^{\overline{k_n}}$ satisfying Equation~\ref{e:invt poly}.
\end{corollary}

\section{Description of the Orbit Space}\label{s:orbit space}

Let $\SS^1$ act on $\CC^m$ linearly and effectively with positive weights $\alpha_1,\dots,\alpha_m$.  The purpose of this section is to study two main features of the differential structure of the orbit space $\CC^m/\SS^1$: the link and the distinguished stratum.

Since the stabilizers of the action away from the origin are proper subgroups of $\SS^1$, we immediately have the following fact.

\begin{proposition}\label{p:orbit space}
Let $\SS^1$ act on $\CC^m$ linearly and effectively.  Then $(\CC^m\smallsetminus\{0\})/\SS^1$ is an orbifold diffeomorphic to $(\SS^{2m-1}/\SS^1)\times(0,\infty)$
\end{proposition}

\begin{remark}\label{r:orbit space}
If the action is not effective, one no longer has an effective orbifold.  However, there remains a diffeomorphism between the two differential spaces.
\end{remark}

In the case of all positive weights, the link $\SS^{2m-1}/\SS^1$ is a well-known orbifold called a \textbf{weighted projective space} $\CC\PP(\alpha_1,\dots,\alpha_m)$.  Although typically a weighted projective space is considered with its complex structure, we discard that here and consider the corresponding differential subspace structure induced by $\CC^m/\SS^1$.  As discussed in Subsection~\ref{ss:lin circ actions}, for fixed $j$, the stabilizer at each point $(0,\dots,0,z_j,0,\dots,0)$ where $z_j\neq 0$ is $\ZZ_{\alpha_j}$.  Any $1$-dimensional orbit-type stratum in $\SS^{2m-1}$ is equal to one of these sets intersected with $\SS^{2m-1}$.  Hence any $0$-dimensional stratum of the corresponding weighted projective space has isotropy group $\ZZ_{\alpha_j}$.  Similar statements can be obtained for each higher dimensional stratum using Subsection~\ref{ss:lin circ actions}; this will be crucial in the proof of Theorem~\ref{t:main}.

The distinguished stratum is the image of the origin, the unique fixed point of the action, via the quotient map (again, assuming all weights are positive).  For general compact linear actions, the differential structure near such distinguished strata is interesting and important.  It detects invariants there that are not topological.  For example, consider $\ZZ_n$ acting on $\CC$ by rotations.  For each $n$, the orbit space is homeomorphic to the plane; however, the differential structure detects the so-called Milnor number (also known as a germ codimension) at the distinguished stratum from which the number $n$ can be recovered (see \cite[Section 5]{watts-orb} for more details).  The Milnor number makes rigorous what can be interpreted in loose terms as ``how fast'' the ``cone'' converges to its apex, without the use of any type of Riemannian metric.  For instance, if we intersect $\CC/\ZZ_n$ as a differential subspace of $\RR^3$ with a plane through the distinguished stratum containing the axis of symmetry, we obtain a curve with a singularity diffeomorphic to the graph of $y^2=x^n$ in $\RR^2$ ($x>0$).  Going back to the $\SS^1$-action, we check for similar behavior.  

Recall our notation $$S_j:=\{(0,\dots,0,z_j,0,\dots,0)\mid z_j\neq 0\}.$$

\begin{proposition}\label{p:1-dim strat}
Let $\SS^1$ act on $\CC^m$ linearly and effectively with positive weights $\alpha_1,\dots,\alpha_m$.  There exists a choice of Hilbert map $\sigma=(\sigma_1,\dots,\sigma_n)$, where the polynomials $\sigma_j$ are of the form in Corollary~\ref{c:invt poly}, such that the image of $\bigcup_jS_j$ under $\sigma$ is closed under scalar multiplication by positive real numbers, forming the non-negative parts of $m$ coordinate axes of $\RR^n$ with the $0$-dimensional stratum at the origin.  Moreover, the $0$- and $1$-dimensional orbit-type strata of $\CC^m/\SS^1$ are contained in these axes.
\end{proposition}

\begin{proof}
Choose the generating set $\{\sigma_1,\dots,\sigma_n\}$ to contain the polynomials $|z_i|^2$ for $i=1,\dots,m$, but no powers of these polynomials greater than $1$.   Then for each $k\in\{1,\dots,n\}$, each polynomial $\sigma_k$ when restricted to $S_j$ is identically $0$ unless it is equal to $|z_j|^2$.  The result follows.
\end{proof}

As an example, consider the case $m=2$, $\alpha_1=1$, and $\alpha_2=2$.  The orbit space is diffeomorphic to the semi-algebraic set in $\RR^4$ given by
\begin{gather*}
y_1\geq 0,\\
y_2\geq 0,\\
y_3^2 + y_4^2 = y_1^2y_2.
\end{gather*}
(See Example~\ref{x:dim2} for details.)
Intersecting this differential subspace with the plane $y_3=y_4=0$, we obtain the non-negative $y_1$- and $y_2$-axes, which together form a curve with differential structure diffeomorphic to the graph of the absolute value function.  The intersection with other planes yields different singularities, however.  For example, intersecting with $y_3=y_1-y_2=0$ yields the curve $y_4^2=y_1^3$, which has a more severe cusp.

\section{Linear $\SS^1$-Actions}\label{s:linear case}

A more sophisticated version of Theorem~\ref{t:main} is below, along with its proof.  We develop an algorithm in the proof for finding the weights of an effective linear circle action.  Examples~\ref{x:dim3} and \ref{x:dim5} illustrate this algorithm.

Let $\SS^1$ act linearly on $\RR^k$ and $\RR^\ell$, and let $\psi\colon\RR^k\to\RR^\ell$ be a smooth $\SS^1$-equivariant map.  Then $\psi$ descends to a map $\widehat{\psi}\colon\RR^k/\SS^1\to\RR^\ell/\SS^1$ making the following diagram commute:

$$\xymatrix{
\RR^k \ar[r]^{\psi} \ar[d]_{\pi_k} & \RR^\ell \ar[d]^{\pi_\ell} \\
\RR^k/\SS^1 \ar[r]_{\widehat{\psi}} & \RR^\ell/\SS^1.\\
}$$

Note that $\widehat{\psi}$ is smooth.  Indeed, let $f\in\CIN(\RR^\ell/\SS^1)$.  It is sufficient to show that $(\widehat{\psi}\circ\pi_k)^*f\in\CIN(\RR^k)$.  But $(\widehat{\psi}\circ\pi_k)^*f = (\pi_\ell\circ\psi)^*f$.  Since $\pi_\ell$ and $\psi$ are smooth, we conclude that $(\widehat{\psi}\circ\pi_k)^*f\in\CIN(\RR^k)$.  This shows that smooth equivariant maps between $\SS^1$-representations are sent to smooth maps between orbit spaces.  In fact, this is a functor to differential spaces.  Theorem~\ref{t:linear case} states that this functor is essentially injective.  We say that a functor $F$ is \textbf{essentially injective} if given objects $c_1$ and $c_2$ in its domain category, $F(c_1)\cong F(c_2)$ implies $c_1\cong c_2$; that is, it is injective on objects up to isomorphism.

\begin{theorem}[Linear Circle Actions]\label{t:linear case}
Let $\mathcal{C}$ be the category of all effective linear actions of $\SS^1$ on finite-dimensional real vector spaces with smooth $\SS^1$-equivariant maps between them.  Then the functor from $\mathcal{C}$ to differential spaces sending such an $\SS^1$-action on a vector space $V$ to the differential space $(V/\SS^1,\CIN(V/\SS^1))$, and sending smooth $\SS^1$-equivariant maps to smooth maps between orbit spaces, is essentially injective on objects.
\end{theorem}

\begin{remark}\labell{r:linear case}
To obtain Theorem~\ref{t:main} from Theorem~\ref{t:linear case}, we need to show that if there is an $\SS^1$-equivariant diffeomorphism $\varphi$ between two $\SS^1$-representations in $\mathcal{C}$, then there is also an $\SS^1$-equivariant linear isomorphism between them.  This follows from the fact that the actions of $\SS^1$ are linear, and so we may identify any such representation with its tangent space at a fixed point with the induced action.  Since $\varphi$ maps the origin to another fixed point, the differential at the origin $d\varphi|_0$ satisfies what is required.
\end{remark}

\begin{proof}
Let $V$ be an $\SS^1$-representation.  Let $\dim V=n$, and identify $V$ with $\RR^n$.  As mentioned previously, the action of $\SS^1$ on $\RR^n$ will always be isomorphic to an $\SS^1$-action on the product $\RR^{n-2m}\times\CC^m$, where $\SS^1$ acts on $\RR^{n-2m}$ trivially, and on $\CC^m$ by Equation~\eqref{e:circle action} such that the weights $\alpha_j$ are positive and $\gcd(\alpha_1,\dots,\alpha_m)=1$.
To complete the proof, we need to obtain the integers $\alpha_1,\dots,\alpha_m$, as well as the dimension $n-2m$ of the trivial representation, from $\CIN((\RR^{n-2m}\times\CC^m)/\SS^1)$.

The quotient topology on $(\RR^{n-2m}\times\CC^m)/\SS^1$ is equal to the initial topology induced by $\CIN((\RR^{n-2m}\times\CC^m)/\SS^1)$ by Lemma~\ref{l:topologies}.  The dimension of $(\RR^{n-2m}\times\CC^m)/\SS^1$, which is $n-1$, is a topological invariant: it is the topological dimension at generic points of the space.  So, the differential structure identifies that the dimension of the $\SS^1$-representation is $n$.

Since $\SS^1$ acts trivially on $\RR^{n-2m}$, the coordinate functions on this factor can be chosen as generators in a generating set of invariant polynomials on $\RR^{n-2m}\times\CC^m$.  It follows that $(\RR^{n-2m}\times\CC^m)/\SS^1$ is diffeomorphic to $\RR^{n-2m}\times(\CC^m/\SS^1)$, and we shall identify the two spaces.

By Theorem~\ref{t:ot strat}, the orbit-type stratification on $\RR^{n-2m}\times\CC^m/\SS^1$ can be obtained from $\CIN(\RR^{n-2m}\times\CC^m/\SS^1)$.  In particular, since $\CC^m$ contains a unique $\SS^1$-fixed point, the orbit space $\RR^{n-2m}\times\CC^m/\SS^1$ has a unique stratum $A$ of minimal dimensional equal to $n-2m$.  It follows that the differential structure identifies the dimension of the trivial representation of $\SS^1$ on $\RR^{n-2m}$.  We now only need to find the weights $\alpha_1,\dots,\alpha_m$ to complete the proof.  Note that we know ahead of time that there are $m$ such weights, as we know the orbit space came from a linear circle action on $\RR^n$ with a trivial factor $\RR^{n-2m}$ of maximal dimension: we can derive $m$ from the two numbers $n$ and $n-2m$.

Removing $A$ from $\RR^{n-2m}\times\CC^m/\SS^1$, we are left with an orbit space of a  locally free $\SS^1$-action, \emph{i.e.\ }an orbifold, by Proposition~\ref{p:orbit space}.  By Corollary~\ref{c:orbifold}, the orders of the isotropy groups at each stratum of this orbit space can be obtained from its differential structure, and hence from $\CIN(\RR^{n-2m}\times\CC^m/\SS^1)$.  Since these finite orders will correspond to subgroups of $\SS^1$, we conclude that those isotropy groups are cyclic groups of the obtained orders.

We claim that there is a natural way to pick out the weights from the orders of these isotropy groups.  That the weights appear at all among these orders is clear: the weight $\alpha_j$ is the order of the stabilizer of all points in
$$S_j = \{(0,\dots,0,z_j,0,\dots,0)\mid z_j\neq 0\}\subseteq\CC^m.$$

The weights completely determine the orbit-type stratification of $\CC^m$ (see Subsection~\ref{ss:lin circ actions}), and hence of $\RR^{n-2m}\times\CC^m$, and its orbit space.  In fact, the orbit-type strata on the orbit space will be unions of images of the sets in Table~\ref{t:s-sets} via the quotient map, and so we can also use the simplex to organize the strata of the orbit space.  Since we know ahead of time that the orbit space is the result of an effective linear circle action, we take advantage of this knowledge and now produce an algorithm on the orbit space starting with the differential structure $\CIN(\RR^{n-2m}\times\CC^m/\SS^1)$ which obtains the weights.

Remove $A$ from $\RR^{n-2m}\times\CC^m/\SS^1$, and denote the collection of the remaining strata by $\mathfrak{S}$.  Equip $\mathfrak{S}$ with a partial ordering $\preceq$, defined by: $S\preceq T$ if $\overline{T}\cap S\neq\emptyset$, in which case $S\subseteq\overline{T}$.  Note that $\mathfrak{S}$ is finite, with open and dense stratum $\calO$ as the top stratum, meaning it is maximal with respect to $\preceq$.  Also, if $S\preceq T$, then either $S=T$ or $\dim(S) < \dim(T)$.  The \textbf{depth} of a stratum $S$ is the number of distinct strata in a maximal chain with $S$ as its minimal element:\\
$$ \depth(S) := \sup\{\nu\mid S=S_0\prec S_1\prec \cdots \prec S_\nu=\calO\}.$$

Start with an element $\calR$ of $(\mathfrak{S},\preceq)$ of maximum depth, and denote its codimension by $2r$; this will be a union of images of sets from Table~\ref{t:s-sets} via the quotient map.  $\calR$ is represented by an $(m-1-r)$-face in the simplex mentioned above; equivalently, it is the union of the image via the quotient map of a codimension-$(2r)$ set in Table~\ref{t:s-sets} along with some lower dimensional such images in its boundary.  In particular, we know $m-r$ vertices are contained in this face; as $\calR$ is minimal with respect to $\preceq$ we know that the $m-r$ corresponding images of sets $S_j$ via the quotient map are contained in $\calR$.  Since the isotropy groups at all points of $\calR$ share the same order, and as mentioned above we know what this order is, we conclude that we know what the order is at the $m-r$ vertices.  So we have obtained $m-r$ weights.
Repeat this step for all strata with the same depth as $\calR$, keeping in mind that while a vertex may appear more than once when applying this step to different strata, its associated weight should only be recorded once.

We continue recursively.  Fix a depth $D$.  Suppose for each stratum $\calQ$ of depth greater than $D$ and for each vertex contained in the face associated to $\calQ$, the associated weight is known.  Let $\calP$ be a stratum of depth $D$ and codimension $2p$, which is represented by an $(m-1-p)$-face in the simplex.  Consequently, this face contains $m-p$ vertices.  Each stratum in $\overline{\calP}\smallsetminus\calP$ has depth greater than $D$, and by assumption, we know the weights associated to the vertices in their corresponding faces.  If the total number of these vertices is not $m-p$, then the remaining vertices must be associated to the stratum $\calP$ itself, and we obtain the order of the corresponding isotropy groups, which is constant at all points of $\calP$.  Repeat this for all strata of depth $D$.  The result is that for each stratum of depth at least $D$, and for each vertex contained in the associated faces, the associated weights are known.

Applying this procedure to all strata of incrementally decreasing depth, we eventually reach the top stratum $\calO$.  If we have not obtained $m$ weights at this point, we apply the above argument one more time to obtain the remaining weights, all equal to $1$.

Since the algorithm above considers every possible vertex in the simplex (equivalently every set $S_j$ in Table~\ref{t:s-sets}), it is guaranteed to produce $m$ weights.  We now have enough information to reconstruct the $\SS^1$-action on $\RR^n$.
\end{proof}

\section{Examples}\label{s:examples}

\begin{example}[$\SS^1\circlearrowright\CC$]\label{x:dim1}
Consider the action of $\SS^1$ on $\CC$, given by $e^{i\theta}\cdot z=e^{i\theta}z$.  It follows from Proposition~\ref{p:invt poly} that $|z|^2$ generates all invariant polynomials.  Thus, the orbit space is identified with the closed interval $[0,\infty)\subset\RR$.  The orbit-type strata in $\CC$ are the origin $\{0\}$ and its complement $\CC\smallsetminus\{0\}$.  
\eoe
\end{example}

\begin{example}[$\SS^1\circlearrowright\CC^2$]\label{x:dim2}
We compute a generating set of invariant polynomials with their relations, as well as the stabilizer groups, of $\SS^1\circlearrowright \CC^2$ with weights $\alpha_1$ and $\alpha_2$. We will assume that the action is effective, and so $\gcd(\alpha_1,\alpha_2)=1$.  The invariant polynomials can be obtained using Corollary~\ref{c:invt poly}.
\begin{align*}
p_1(z_1,z_2) &= |z_1|^2\\
p_2(z_1,z_2) &= |z_2|^2\\
p_3(z_1,z_2) &= \mathfrak{Re}(z_1^{\alpha_2}\overline{z}_2^{\alpha_1})\\
p_4(z_1,z_2) &= \mathfrak{Im}(z_1^{\alpha_2}\overline{z}_2^{\alpha_1})
\end{align*}

The relations can be verified using a Gr\"obner basis \cite[Chapter 1]{sturmfels} and \cite{PS} .
\begin{align*}
p_1&\geq 0\\
p_2&\geq 0\\
p_3^2+p_4^2&=p_1^{\alpha_2}p_2^{\alpha_1}
\end{align*}

The stabilizer groups can be computed directly:
\begin{align*}
\Stab(0,0)&=\SS^1, \\
\Stab(z_1,0)&=\ZZ_{\alpha_1} && \text{($z_1\neq 0$)},\\
\Stab(0,z_2)&=\ZZ_{\alpha_2} && \text{($z_2\neq 0)$},\\
\Stab(z_1,z_2)&=\{1\} && \text{elsewhere}. \\
\end{align*}
Here, in the case that $\alpha_1$ or $\alpha_2$ is $1$, we define $\ZZ_1$ to be the trivial group.
\eoe
\end{example}

\begin{example}[$\SS^1\circlearrowright\CC^3$]\label{x:dim3}
We illustrate the algorithm used in the proof of Theorem~\ref{t:linear case} for a simple example. Consider $\SS^1$ acting on $\CC^3$ linearly and effectively with weights $1$, $2$, and $3$.  

We find the orbit-type strata of the orbit space by constructing a table similar to Table~\ref{t:s-sets}; we do so in Table~\ref{t:dim3}.  The orbit-type strata are the distinguished stratum, the open dense stratum
$$\calO = \pi(S_{123}\cup S_{12} \cup S_{13} \cup S_{23} \cup S_1)$$
and two $1$-dimensional strata $\calP_1 = \pi(S_2)$ and $\calP_2 = \pi(S_3)$, with associated orders of isotropy groups $2$ and $3$, respectively.
\begin{center}
\begin{table}
\begin{tabular}{|c|c|c|}
\hline
\underline{Set} & \underline{Order of Stabilizer} & \underline{Codimension} \\
\hline
$S_{123}$ & $1$ & $0$ \\
\hline
$S_{12}$ & $1$ & $2$ \\
\hline
$S_{13}$ & $1$ & $2$  \\
\hline
$S_{23}$ & $1$ & $2$  \\
\hline
$S_{1}$ & $1$ & $4$  \\
\hline
$S_{2}$ & $2$ & $4$  \\
\hline
$S_{3}$ & $3$ & $4$  \\
\hline
\end{tabular}
\caption{}
\label{t:dim3}
\end{table}
\end{center}

The Hasse diagram for the partial order $\preceq$ as introduced in the proof of Theorem~\ref{t:linear case} is as follows.
$$
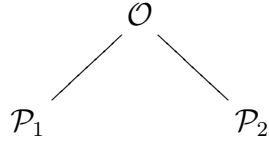
\xymatrix{
 & \calO \ar@{-}[dl] \ar@{-}[dr] & \\
\calP_1 & & \calP_2 \\
}$$
Stratum $\calP_1$ has codimension $4$, and so corresponds to a vertex of the simplex described in Subsection~\ref{ss:lin circ actions} (or equivalently a set $S_j$ in Table~\ref{t:dim3}); it has associated order $2$, which is one of the weights.  Similarly $\calP_2$ has associated order $3$, another weight.  We are expecting $3$ weights in total, and so the remaining weight must be the order associated to $\calO$, which is $1$.
\eoe
\end{example}

\begin{example}[$\SS^1\circlearrowright\CC^5$]\label{x:dim5}
We illustrate the algorithm used in the proof of Theorem~\ref{t:linear case} for a more complicated action.  Let $\SS^1$ act linearly and effectively on $\CC^5$ with weights $2$, $2$, $3$, $4$, and $6$.

To find the orbit-type strata of the orbit space, we could construct a table similar to Table~\ref{t:s-sets}; we do not to save space.  The resulting strata are, besides the distinguished stratum:
\begin{align*}
\calO =&~ \pi(S_{12345}\cup S_{1234}\cup S_{1235}\cup S_{1345}\cup S_{2345}\cup S_{123}\cup S_{134}\\
&~ \cup S_{135}\cup S_{234}\cup S_{235}\cup S_{345}\cup S_{13}\cup S_{23}\cup S_{34})\\
\calP =&~ \pi(S_{1245}\cup S_{124}\cup S_{125}\cup S_{145}\cup S_{245}\cup S_{12}\cup S_{14}\\
&~ \cup S_{15}\cup S_{24}\cup S_{25}\cup S_{45}\cup S_1\cup S_2)\\
\calQ =&~ \pi(S_{35}\cup S_3)\\
\calR_1 =&~ \pi(S_4)\\
\calR_2 =&~ \pi(S_5)
\end{align*}

The associated orders of isotropy groups are:
$$\calO\!:1\quad \calP\!:2\quad \calQ\!:3\quad \calR_1\!:4\quad \calR_2\!:6$$

The Hasse diagram for the partial order $\preceq$ used in the proof of Theorem~\ref{t:linear case} is:
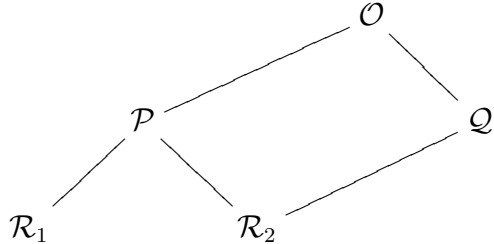
$$\xymatrix{
 & & & \calO \ar@{-}[dll] \ar@{-}[dr] & \\
 & \calP \ar@{-}[dl] \ar@{-}[dr] & & & \calQ \ar@{-}[dll] \\
\calR_1 & & \calR_2 & & \\
}$$
Stratum $\calR_1$ has codimension $8$, and so corresponds to a vertex of the simplex described in Subsection~\ref{ss:lin circ actions} (or equivalently a set $S_j$ in Table~\ref{t:s-sets}); it has associated order $4$, which is one of the weights.  Similarly, stratum $\calR_2$ also yields a weight; namely, $6$.  $\calP$ has codimension $2$, and therefore it corresponds to a $3$-face in the simplex (equivalently, a set $S_{j_1j_2j_3j_4}$ in Table~\ref{t:s-sets}), and so its closure contains $4$ vertices.  Two of the weights have been found, and so the other two must be associated to $\calP$ itself, which has order $2$.  We now have weights $2$, $2$, $4$, and $6$.  $\calQ$ has codimension $6$, and so corresponds to an edge in the simplex (equivalently, a set $S_{j_1j_2}$).  There are two vertices in its closure, one of which corresponds to $\calR_2$.  So the other weight must be the order associated to $\calQ$, which is $3$.  Since we now have five weights, we are done.
\eoe
\end{example}


\end{document}